\documentclass{amsart}
\pdfoutput=1
\usepackage[margin=1.2in]{geometry} 
\usepackage[utf8]{inputenc}
\usepackage[T1]{fontenc}
\usepackage{amsmath, amsthm, amssymb, cancel, enumerate, mathrsfs, mathtools}
\usepackage{setspace}                   
\usepackage{indentfirst}                
\usepackage{type1cm}                    
\usepackage{titletoc}
\usepackage{dsfont} 
\usepackage{physics}
\usepackage{subcaption}
\usepackage{enumerate}
\usepackage{fancybox}
\usepackage{wasysym}
\usepackage{comment}
\usepackage{longtable}
\usepackage{bm}
\usepackage{upgreek}
\usepackage{tikz}
\usepackage[square,sort,comma,numbers]{natbib}
\usepackage{ textcomp }
\usepackage{url}

\graphicspath{{./figuras/}}             
\fontsize{60}{62}\usefont{OT1}{cmr}{m}{n}{\selectfont}
\cleardoublepage
\normalsize

\theoremstyle{definition}
\newtheorem{Def}{\textsc{Definition}}[section]

\newtheorem{Rem}{\textsc{Remark}}

\theoremstyle{plain}
\newtheorem{Prop}[Def]{\textsc{Proposition}}
\newtheorem{Ques}[Def]{\textsc{Question}}
\newtheorem{Lema}[Def]{\textsc{Lemma}}
\newtheorem{Corol}[Def]{\textsc{Corollary}}
\newtheorem{Teo}[Def]{\textsc{Theorem}}
\newtheorem{Cla}{Claim}

\newcommand{\restr}[2]{\left. #1 \right|_{#2}}

\newcommand{\rng}{\text{rng}}

\makeatletter
\def\moverlay{\mathpalette\mov@rlay}
\def\mov@rlay#1#2{\leavevmode\vtop{%
   \baselineskip\z@skip \lineskiplimit-\maxdimen
   \ialign{\hfil$\m@th#1##$\hfil\cr#2\crcr}}}
\newcommand{\charfusion}[3][\mathord]{
    #1{\ifx#1\mathop\vphantom{#2}\fi
        \mathpalette\mov@rlay{#2\cr#3}
      }
    \ifx#1\mathop\expandafter\displaylimits\fi}
\makeatother

\newcommand{\Hom}{\text{Hom}([2^{\mathfrak{c}}]^{< \omega},2)}
\newcommand{\homm}{\text{Hom}([\omega]^{< \omega},2)}

\usepackage{scalerel,stackengine}
\stackMath
\newcommand\reallywidehat[1]{%
\savestack{\tmpbox}{\stretchto{%
  \scaleto{%
    \scalerel*[\widthof{\ensuremath{#1}}]{\kern-.6pt\bigwedge\kern-.6pt}%
    {\rule[-\textheight/2]{1ex}{\textheight}}
  }{\textheight}%
}{0.5ex}}%
\stackon[1pt]{#1}{\tmpbox}%
}
\usepackage[utf8]{inputenc}

\title{Some pseudocompact-like properties in certain topological groups}
\author{Artur Hideyuki Tomita$^{\dag}$}
\thanks{$^{\dag}$ The first listed author has received financial support from FAPESP 2021/00177-4}
\author{Juliane Trianon-Fraga$^{\ddag}$}
\thanks{$^{\ddag}$ The second listed author has received financial support from FAPESP 2019/12628-0.}
\newcommand{\Addresses}{
  \bigskip
  
\indent \textsc{Depto de Matemática, Instituto de Matemática e Estatística, Universidade de São Paulo, Rua do Matão, 1010, CEP 05508-090, São Paulo, SP, Brazil.}\par\nopagebreak
  \textit{E-mail addresses:} tomita@ime.usp.br (A. H. Tomita), jtrianon@ime.usp.br (J. Trianon-Fraga).}

\begin{document}
\begin{abstract}
   We construct in ZFC a countably compact group without non-trivial convergent sequences of size $2^{\mathfrak{c}}$, answering a question of Bellini, Rodrigues and Tomita \cite{tomvinmat}. We also construct in ZFC a selectively pseudocompact group which is not countably pracompact, showing that these two properties are not equivalent in the class of topological groups. Using the same technique, we construct a group which has all powers selectively pseudocompact but is not countably pracompact, assuming the existence of a selective ultrafilter.
   
\end{abstract}
\smallskip
   \maketitle

\section{Introduction}

In this paper, every topological space will be Tychonoff (Hausdorff and completely regular) and every topological group will be Hausdorff (thus, also Tychonoff). For an infinite set $X$, $[X]^{<\omega}$ will denote the family of all finite subsets of $X$, and $[X]^{\omega}$ will denote the family of all countable subsets of $X$. Recall that an infinite topological space $X$ is said to be
\begin{itemize}
    \item \textit{pseudocompact} if each continuous real-valued function on $X$ is bounded;
    \item \textit{countably compact} if every infinite subset of $X$ has an accumulation point in $X$;
    \item \textit{countably pracompact} if there exists a dense subset $D$ in $X$ such that every infinite subset of $D$ has an accumulation point in $X$.
\end{itemize}

 We denote the set of non-principal (free) ultrafilters on $\omega$ by $\omega^*$. Given $p \in \omega^*$, $x \in X$ and a sequence $\{x_n: n \in \omega\} \subset X$, following \cite{Bernstein}, we say that $x$ is a $p-$limit point of $(x_n)_{n \in \omega}$ if, for every neighborhood $U$ of $x$, we have that $\{n \in \omega: x_n \in U\} \in p$. If $X$ is Hausdorff, a sequence $\{x_n: n \in \omega\}$ has at most one $p$-limit point $x$ and we write $x=p-\lim_{n \in \omega}x_n$. We may consider a topology on $\omega^*$ which has, for each $A \subset \omega$, the sets $A^* \doteq \{p \in \omega^*: A \in p\}$ as basic open sets. Also, we recall that $p \in \omega^*$ is a $P-$\textit{point} if whenever $\{V_n: n \in \omega\}$ is a countable family of neighborhoods of $p$, $p \in  \text{int}\big( \bigcap_{n \in \omega} V_n\big)$, and $p \in \omega^*$ is a \textit{weak} $P-$\textit{point} if it is not an accumulation point of any countable subset of $\omega^*$. Kunen showed in ZFC that there exists $2^{\mathfrak{c}}$ points in $\omega^*$ which are weak $P-$points but not $P$-points \cite{kunen}. 
 
 It is not hard to show that $x \in X$ is an accumulation point of a sequence $\{x_n: n \in \omega\} \subset X$ if and only if there exists $p \in \omega^*$ such that $x= p-\lim_{n \in \omega}x_n$. Thus, it follows easily that
\begin{itemize}
\item $X$ is countably compact if and only if every sequence $\{x_n:n \in \omega\} \subset X$ has a $p-$limit, for some $p \in \omega^*$.
    \item $X$ is countably pracompact if and only if there exists a dense subset $D$ in $X$ such that every sequence $\{x_n: n \in \omega\} \subset D$ has a $p-$limit in $X$, for some $p \in \omega^*$.
\end{itemize} 
For pseudocompact spaces, a similar equivalence holds: $X$ is pseudocompact if and only if for every countable family $\{U_n: n \in \omega\}$ of nonempty open sets of $X$, there exists $x \in X$ and $p \in \omega^*$ such that, for each neighborhood $V$ of $x$, $\{n \in \omega: V \cap U_n \neq \emptyset\} \in p$.
 
Since the introduction of pseudocompactness by Hewitt \cite{hewitt}, many related concepts have emerged, which provide new topological spaces, with different properties:

\begin{Def}
Let $X$ be a topological space.
\begin{enumerate}[(1)]
\item For $p \in \omega^*$, $X$ is called $p-$compact if every sequence of points in $X$ has a $p-$limit. 
    \item For $p \in \omega^*$, $X$ is called $p-$\textit{pseudocompact} if for every countable family $\{U_n: n \in \omega\}$ of nonempty open sets of $X$, there exists $x \in X$ such that, for each neighborhood $V$ of $x$, $\{n \in \omega: V \cap U_n \neq \emptyset\} \in p$.
    \item $X$ is called \textit{ultrapseudocompact} if $X$ is $p-$pseudocompact for every $p \in \omega^*$.
    
    \item $X$ is called \textit{selectively pseudocompact}\footnotemark \ if for each sequence $(U_n)_{n \in \omega}$ of nonempty open subsets of $X$ there is a sequence $\{x_n: n \in \omega\} \subset X$, $x \in X$ and $p \in \omega^*$ such that $x=p-\lim_{n \in \omega}x_n$ and, for each $n \in \omega$, $x_n \in U_n$.
\footnotetext{This concept was originally defined in \cite{yasser2} under the name \textit{strong pseudocompactness}, but later the name was changed, since there were already two different properties named in the previous way (in \cite{propriedade1} and \cite{propriedade2}).}
    \item For $p \in \omega^*$, $X$ is called \textit{selectively} $p-$\textit{pseudocompact}\footnotemark \ if, for each sequence $(U_n)_{n \in \omega}$ of nonempty open subsets of $X$, there is a sequence $(x_n)_{n \in \omega}$ of points in $X$ and $x \in X$ such that $x=p-\lim_{n \in \omega} x_n$ and, for each $n \in \omega$, $x_n \in U_n$.

    \end{enumerate}
\end{Def}

 \footnotetext{Similarly, this concept was defined originally under the name \textit{strong }$p$\textit{-pseudocompactness.}}

The concept of $p-$compactness was introduced in \cite{Bernstein}, $p-$pseudocompactness and ultrapseudocompactness in \cite{ginsa}, selective $p$-pseudocompactness in \cite{yasser1} and selective pseudocompactness in \cite{yasser2}. Regarding these notions, it follows straight from the definitions that for each $p \in \omega^*$,
~\\
\begin{equation*}
p-\text{compactness} \Rightarrow \text{selective }p-\text{pseudocompactness} \Rightarrow p-\text{pseudocompactness} \Rightarrow \text{pseudocompactness}
\end{equation*}
\begin{center}and\end{center}
\begin{equation*}
\text{ultrapseudocompactness} \Rightarrow p-\text{pseudocompactness}.
\end{equation*}
For each $p \in \omega^*$, \cite{yasser1} shows an example of a $p-$pseudocompact space which is not ultrapseudocompact. Also it is evident that
\begin{align*}
&\text{countable compactness} \Rightarrow \text{countable pracompactness} \Rightarrow \\ &\Rightarrow \text{selective pseudocompactness} \Rightarrow  \text{pseudocompactness}.
\end{align*} 
In \cite{Gillman}, there is an example of a countably compact space which is not $p-$pseudocompact for any $p \in \omega^*$, and in \cite{salvador}  an example of an ultrapseudocompact space which is not selectively pseudocompact. 

For topological groups we can say more:
\begin{Teo}[\cite{garchis}]
For a topological group $G$, the following conditions are equivalent.
\begin{enumerate}[(1)]
    \item $G$ is pseudocompact.
    \item There is a $p \in \omega^*$ such that $G$ is $p-$pseudocompact.
    \item $G$ is ultrapseudocompact.
\end{enumerate}
\end{Teo}
\noindent In \cite{tomita1}, there is an example of a selectively pseudocompact group which is not countably compact. The question whether pseudocompactness implies selective pseudocompactness in topological groups was posed in \cite{yasser2}, and solved by Garcia-Ferreira and Tomita, who proved that there exists a pseudocompact group which is not selectively pseudocompact \cite{tomita1}. Hence, the selective pseudocompactness is not another equivalent notion for pseudocompactness in topological groups. In this paper we prove that there exists a topological group which is selectively pseudocompact but is not countably pracompact. Therefore, the notion of countably pracompactness is even more strict in topological groups.
\vspace{0.5cm}

In what follows, we recall some basic facts about selective ultrafilters.
\begin{Def}
A selective ultrafilter on $\omega$ is a free ultrafilter $p$ on $\omega$ such that for every partition $\{A_n: n \in \omega\}$ of $\omega$, either there exists $n \in \omega$ such that $A_n \in p$ or there exists $B \in p$ such that $|B \cap A_n| =1$ for every $n \in \omega$.
\end{Def}
When handling the combinatorial properties of selective ultrafilters, it is often useful to use some of their equivalent properties, like those given by the well know proposition below (for a proof, see \cite{negre}, for instance).
\begin{Prop}
Let $p \in \omega^*$. The following are equivalent.
\begin{enumerate}[a)]
    \item $p$ is a selective ultrafilter.
    \item For every $f \in \omega^{\omega}$, there exists $A \in p$ such that $\restr{f}{A}$ is either constant or one-to-one.
    \item For every function $f: [\omega]^{2} \to 2$ there exists $A \in p$ such that $\restr{f}{[A]^2}$ is constant.
\end{enumerate}
\end{Prop}
The existence of selective ultrafilters is independent of ZFC. In fact, Martin's Axiom (MA) implies the existence of $2^{\mathfrak{c}}$ selective ultrafilters \cite{Blass}, while there is a model of ZFC in which there are no selective ultrafilters \cite{Shelah}.

Assuming the existence of a single selective ultrafilter, we also prove in this paper that there exists a topological group which is not countably pracompact and has all powers selectively pseudocompact.
\vspace{0.5cm}

 The question whether there exists a countably compact group without non-trivial convergent sequences in ZFC has been left open for a long time, and it is related to an old open problem posed by Comfort, after proving, together with Ross, that the product of any family of pseudocompact topological  groups is pseudocompact \cite{comross}. Comfort asked whether there are countably compact topological groups whose product is not countably compact\footnotemark. The first consistent positive answer was given by Douwen, under MA \cite{Douwen}. More specifically, Douwen proved the two following lemmas.
 
 \footnotetext{More generally, he asked in the survey book \textit{Open Problems
in Topology} whether there is, for every (not necessarily infinite) cardinal number $\alpha  \leq 2^{\mathfrak{c}}$, a topological group $G$ such that $G^{\gamma}$ is countably compact for all cardinals $\gamma < \alpha$, but $G^{\alpha}$ is
not countably compact.}
 
 \begin{Lema}[\cite{Douwen}]\label{com1}
 (ZFC) Every infinite Boolean countably compact group without non-trivial convergent sequences contains two countably compact subgroups whose product is not countably compact.
 \end{Lema}
Tomita proved in ZFC another version of Lemma \ref{com1}: the existence of a countably compact Abelian group without non-trivial convergent sequences implies the existence of a countably compact group whose square is not countably compact \cite{DouTom1}. Also, a version for finite and countable powers (for torsion groups) and finite powers (for non-torsion groups) appears in \cite{tomitaTopAppl2019ZFC}.
 
 \begin{Lema}[\cite{Douwen}]\label{com2}
(MA) There exists an infinite Boolean countably compact group without non-trivial convergent sequences.
 \end{Lema}
 
Together, the two lemmas above show that, under MA, there exists two countably compact groups whose product is not countably compact. Some years later, a more general Comfort's question (see the footnote in this page) was solved for every cardinal $\kappa \leq 2^{\mathfrak{c}}$, assuming the existence of $2^{\mathfrak{c}}$ selective ultrafilters and that $2^{\mathfrak{c}}=2^{<2^{\mathfrak{c}}}$ \cite{tomita34}. Also, in \cite{tomita2015} it was constructed the first torsion free example of a topological group whose least power that fails to be countably compact is $\omega$, using $\mathfrak{c}$ incomparable (according to the Rudin-Keisler ordering in $\omega^*$) selective ultrafilters.

Using tools outside ZFC, many other examples of countably compact groups without non-trivial convergent sequences were given over the years. The first one appeared in \cite{Haj}, under CH. In \cite{koszmider&tomita&watson}, an example was obtained from Martin's Axiom for countable posets, and in \cite{selective} from a single selective ultrafilter, improving the technique, since CH and MA imply the existence of selective ultrafilters. Nevertheless,  \cite{szeptycki&tomita} showed that the existence of such groups does not imply the existence of selective ultrafilters. It was left open for a long time whether there exists an example in ZFC. Finally, in 2021, Hrušák, van Mill, Ramos-García, and Shelah \cite{Michael} proved that:
\begin{Teo}[\cite{Michael}] \label{mich}
In ZFC, there exists a Hausdorff countably compact topological Boolean group (of size $\mathfrak{c}$) without non-trivial convergent sequences.
\end{Teo}
\noindent Due to Lemma \ref{com1}, this result also solves the original Comfort's question.

 In the present article, we prove in ZFC that there exists a Hausdorff countably compact topological Boolean group without non-trivial convergent sequences of size $2^{\mathfrak{c}}$, answering a question posed in \cite{tomvinmat}. We will be dealing with Boolean groups, which are also Boolean vector spaces over the field $2=\{0,1\}$, and thus we can talk about general linear algebra concepts concerning these groups, such as \textit{linearly independent subsets}. More specifically, if $D \subset 2^{\mathfrak{c}}$ is an infinite set, we will consider the Boolean group $[D]^{<\omega}$ with the symmetric difference $\triangle$ as the group operation and $\emptyset$ as the neutral element. Given $p \in \omega^*$, one may define an equivalence relation on $([D]^{<\omega})^{\omega}$ by letting $f \equiv_p g$ iff $\{n \in \omega: f(n)=g(n)\} \in p$. We let $[f]_p$ be the equivalence class determined by $f$ and $([D]^{<\omega})^{\omega}/p$ be $([D]^{<\omega})^{\omega}/\equiv_p$. Notice that this set has a natural vector space structure (over the field $2$). For each $D_0 \in [D]^{<\omega}$, the constant function in $([D]^{<\omega})^{\omega}$ which takes only the value $D_0$ will be denoted by $\vec{D_0}$. If $\alpha < 2^{\mathfrak{c}}$ is an ordinal, then $\vec{\{\alpha\}}$ will be denoted simply by $\vec{\alpha}$.

\section{Auxiliary results}

In this section we present the auxiliary results that we will use in the constructions. We start with a simple fact from linear algebra. 
 \begin{Lema}\label{novo}
    Let $A$, $B$ and $C$ be subsets in a Boolean vector space. Suppose that $A$ is a finite set and that $A \cup C$, $B \cup C$ are linearly independent. Then there exists $B' \subset B$ such that $|B'| \leq |A|$ and $A \cup C \cup (B \setminus B')$ is linearly independent.
    \end{Lema}
    \begin{proof}
    We prove the result by induction on $|A|$. First, suppose that $|A|=1$, that is, $A$ has a single element $a \neq 0$. If $A \cup C \cup B$ is linearly independent, we simply consider $B'= \emptyset$. Otherwise, there is a non-trivial linear combination of elements in $A \cup C \cup B$ that equals zero. Note that $a$ and some element in $B$ must appear in this linear combination, since $A \cup C$ and $B \cup C$ are linearly independent. Thus, we have
    \[ a = Cl(C)_1 \triangle Cl(B)_1, \]
    for some $Cl(B)_1\neq 0$ and $Cl(C)_1$ linear combinations of elements in $B$ and $C$, respectively. Choose an element $b \in B$ which appear in $Cl(B)_1$. We claim that $A \cup C \cup (B \setminus \{b\})$ is linearly independent. Indeed, otherwise we would have
    \[
    a  =  Cl(C)_2 \triangle Cl(B)_2,
    \]
    for some $Cl(B)_2\neq 0$ and $Cl(C)_2$ linear combinations of elements in $B \setminus \{b\}$ and $C$, respectively. But this cannot happen, since $B \cup C$ is linearly independent and $Cl(B)_1 \neq Cl(B)_2$. Hence, we have proved the result if $A$ is a set of size $1$.
    
    Suppose that the Lemma holds for sets of size $n \in \omega$, and that $A$ has a size $n+1$. In this case, letting $a \in A$, we may apply the hypothesis to the sets $A \setminus \{a\}$, $B$ and $C$. Thus, we get $B_0 \subset B$ so that $|B_0| \leq n$ and that $(A \setminus \{a\}) \cup C \cup (B \setminus B_0)$ is linearly independent. Now, we apply the result for sets of size $1$ to the sets $\{a\}$, $B \setminus B_0$ and $(A \setminus \{a\}) \cup C$, and we are done.
 \end{proof}
    
    The next corollary appears in \cite{Michael}.
    \begin{Corol}\label{antigo}
    Let $A$ and $B$ be linearly independent subsets in a Boolean vector space with $A$ a finite set. Then there is $B' \subset B$ such that $|B'| \leq |A|$ and $A \cup (B \setminus B')$ is linearly independent.
    \end{Corol} 
    \begin{proof}
    Use the previous lemma for $C= \emptyset$.
    \end{proof}

The main new idea that appears in \cite{Michael} when proving Theorem \ref{mich} is the use of a clever filter to generate a suitable family of ultrafilters $\{p_{\alpha}:\alpha < \mathfrak{c}\} \subset \omega^*$, given by the next result.

\begin{Prop}[\cite{Michael}]\label{primeiro}
There is a family $\{p_{\alpha}: \alpha < \mathfrak{c}\} \subset \omega^*$ such that, for every $D \in [\mathfrak{c}]^{\omega}$ and $\{f_{\alpha}: \alpha \in D\}$ such that each $f_{\alpha}$ is an one-to-one enumeration of
linearly independent elements of $[\mathfrak{c}]^{< \omega}$, there is a sequence $<U_{\alpha}: \alpha \in D>$ that satisfies
\begin{enumerate}[(i)]
    \item $\{U_{\alpha}: \alpha \in D\}$ is a family of pairwise disjoint subsets of $\omega$;
    \item $U_{\alpha} \in p_{\alpha}$ for every $\alpha \in D$;
    \item $\{f_{\alpha}(n): \alpha \in D \text{ and } n \in U_{\alpha}\}$ is a linearly independent subset of $[\mathfrak{c}]^{< \omega}$.
\end{enumerate}
\end{Prop}
\noindent Using fundamentally the same idea, with just a slight modification, we construct a similar suitable family of ultrafilters $\{p_{\alpha}:\alpha < 2^\mathfrak{c}\} \subset \omega^*$, which permits, in the same way that was done in \cite{Michael}, the construction of a group of size $2^{\mathfrak{c}}$ satisfying Theorem \ref{mich}. For the sake of completeness, we shall not omit any of the arguments. 

\begin{Prop}\label{suitable}
There is a family $\{p_{\alpha}: \alpha < 2^{\mathfrak{c}}\} \subset \omega^*$ such that, for every $D \in [2^{\mathfrak{c}}]^{\omega}$ and $\{f_{\alpha}: \alpha \in D\}$ such that each $f_{\alpha}$ is an one-to-one enumeration of
linearly independent elements of $[2^{\mathfrak{c}}]^{< \omega}$, there is a sequence $<U_{\alpha}: \alpha \in D>$ that satisfies
\begin{enumerate}[I)]
    \item $\{U_{\alpha}: \alpha \in D\}$ is a family of pairwise disjoint subsets of $\omega$;
    \item $U_{\alpha} \in p_{\alpha}$ for every $\alpha \in D$;
    \item $\{f_{\alpha}(n): \alpha \in D \text{ and } n \in U_{\alpha}\}$ is a linearly independent subset of $[2^{\mathfrak{c}}]^{< \omega}$.
\end{enumerate}
\end{Prop}

\begin{proof}
Fix $\{I_n: n \in \omega\}$ a partition of $\omega$ into finite sets such that 
\[
|I_n| > n \cdot \sum_{m<n}|I_m|,
\]
and let
\[
\mathcal{B} = \{B \subset \omega: \forall n \in \omega, |I_n \setminus B| \leq \sum_{m<n}|I_m|\}.
\]
Note that the intersection of every finite subfamily of $\mathcal{B}$ is infinite, thus we may consider the filter $\mathcal{F}$ it generates, which is free. If $A$ is an infinite subset of $\omega$, notice that we have, for every $f \in \mathcal{F}$, $|f \cap \bigcup_{n \in A} I_n| = \omega$.  

\begin{Rem}
In \cite{Michael}, it was fixed at this point an almost disjoint family $(A_{\alpha})_{\alpha < \mathfrak{c}}$ of size $\mathfrak{c}$ of infinite subsets of $\omega$, and each ultrafilter $p_{\alpha}$ was chosen extending $\restr{\mathcal{F}}{\bigcup_{n \in A_{\alpha}} I_n}$. Here, we construct them by using \textit{weak $P$-points} instead. Here lies the difference between the two constructions.
\end{Rem}

Let $\{q_{\xi}: \xi < 2^{\mathfrak{c}}\}$ be the set of \textit{weak $P$-points}\footnotemark. \footnotetext{Recall that Kunen showed that there are $2^{\mathfrak{c}}$ of them in ZFC \cite{kunen}.}
For each $\xi < 2^{\mathfrak{c}}$, we fix a free ultrafilter $p_{\xi}$ containing the set
\[
\left\{f \cap \bigcup_{n \in A} I_n : A \in q_{\xi} \text{, } f \in \mathcal{F}  \right\}.
\]
We shall prove that the family of free ultrafilters $(p_{\xi})_{\xi < 2^{\mathfrak{c}}}$ satisfies the Proposition. For this, we fix a set $D=\{\alpha_n: n \in \omega\} \in [2^{\mathfrak{c}}]^{\omega}$ and a family  $\{f_{\alpha}: \alpha \in D\}$ of one-to-one
sequences of linearly independent elements of $[2^{\mathfrak{c}}]^{< \omega}$. For every $\alpha_n, n \in \omega$, we choose $C_{\alpha_n} \in q_{\alpha_n}$ in a way that $(C_{\alpha_n})_{n \in \omega}$ is a family of pairwise disjoint sets, which can be done, since $(q_{\alpha_n})_{n \in \omega}$ is a sequence of weak $P-$points. Now, let $\{B_n: n \in \omega\}$ be a partition of  $\omega$ such that $B_n =^{*} C_{\alpha_n}$, for every $n \in \omega$ (we could use the sets $C_{\alpha_n}$ directly, but this choice will simplify a bit).

Suppose that, for each $k \in \omega$, $k \in B_{n_k}$. We shall construct a sequence of sets $(R_k)_{k \in \omega}$ satisfying that, for each $k \in \omega$,
\begin{enumerate}[1)]
\item $R_k \subset I_k$;
\item $|R_k| \leq \sum_{m<k}|I_m|$;
\item $f_{\alpha_{n_0}}[I_0 \setminus R_0] \cup f_{\alpha_{n_1}}[I_1 \setminus R_1] \cup... \cup f_{\alpha_{n_k}}[I_k \setminus R_k] \text{ is linearly independent}$.
\end{enumerate}
Let $R_0= \emptyset$, and, given $N > 0$, assume that we have already constructed sets $(R_k)_{k < N}$ satisfying the conditions above for each $k<N$. Since both $f_{\alpha_N}[I_N]$ and $f_{\alpha_{n_0}}[I_0 ] \cup f_{\alpha_{n_1}}[I_1 \setminus R_1] \cup... \cup f_{\alpha_{n_{N-1}}}[I_{N-1} \setminus R_{N-1}] \text{ are linearly independent}$, Corollary \ref{antigo} implies that there exists $R_N \subset I_N$ such that
\[
|R_N| \leq \sum_{m<N}|I_m|
\]
and
\[
f_{\alpha_{n_0}}[I_0] \cup...\cup f_{\alpha_{N_n}}[I_N \setminus R_N] \text{ is linearly independent.}
\]
Therefore, there exists a family $(R_k)_{k \in \omega}$ satisfying the three conditions above, for every $k \in \omega$. Hence, the set
\[
B \doteq \bigcup_{l \in \omega} (I_l \setminus R_l)
\]
satisfy the following:
\begin{enumerate}[i)]
    \item $I_0 \subset B$;
    \item for every $l \in \omega$, $
    |I_l \setminus B| =|R_l| \leq \sum_{m < l}|I_m|;
    $
    \item $\{f_{\alpha_n}(m): n \in \omega \text{, } m \in B \cap I_l \text{, } l \in B_n\}$ is linearly independent.
    \end{enumerate}
It follows from ii) that $B \in \mathcal{B}$, and defining, for each $n \in \omega$,
 \[
 U_n \doteq B \cap \bigcup_{l \in B_n} I_l,
 \]
it is clear that $(U_n)_{n \in \omega}$ is a family of pairwise disjoint sets. Furthermore, it follows from the definition of $p_{\alpha_n}$ and from the fact that $B_n \in q_{\alpha_n}$, that $U_n \in p_{\alpha_n}$. Finally, by construction,
\[
\{f_{\alpha_n}(m): n \in \omega \text{, } m \in U_n\}
\]
is linearly independent.
\end{proof}

For the next result, we fix the family of ultrafilters $\{p_{\alpha}:\alpha<2^{\mathfrak{c}}\}$ constructed in the previous proposition.

\begin{Lema}\label{Homs}
Let $\{f_{\alpha}: \alpha \in I\}$, with $I \subset 2^{\mathfrak{c}}$, be a family of one-to-one enumerations of linearly independent elements of $[2^{\mathfrak{c}}]^{<\omega}$, and $D \in [2^{\mathfrak{c}}]^{\omega}$ be such that, for every $\alpha \in D \cap I$, $\bigcup_{n \in \omega}f_{\alpha}(n) \subset D$. Consider also $D_0 \in [D]^{<\omega}$ and $F:D_0 \to 2$ a function. Then there exists a homomorphism $\phi: [D]^{<\omega} \to 2$ so that, for every $\alpha \in D \cap I$,
\[
\phi(\{\alpha\})=p_{\alpha}-\lim_{n \in \omega}\phi(f_{\alpha}(n)),
\]and, for each $d \in D_0$, $\phi(\{d\})=F(d)$. 
\end{Lema}
\begin{proof}
Enumerate $(D \cap I)$ as $\{\alpha_n: n \in \omega\}$ so that  $\{\alpha_0,...,\alpha_r\} = \displaystyle D_0 \cap I$, and let $D_0 \setminus I = \{d_0,...,d_l\}$. According to Proposition \ref{suitable}, there is a sequence $<U_{\alpha}: \alpha \in D \cap I>$ that satisfies
\begin{enumerate}[I)]

    \item $\{U_{\alpha}: \alpha \in D \cap I\}$ is a family of pairwise disjoint subsets of $\omega$;
    \item $U_{\alpha} \in p_{\alpha}$ for every $\alpha \in D \cap I$;
    \item $\{f_{\alpha}(n): \alpha \in D \cap I \text{ and } n \in U_{\alpha}\}$ is a linearly independent subset of $[2^{\mathfrak{c}}]^{< \omega}$.
\end{enumerate}

Letting 
\[
E_0 \doteq \{\{d_0\},...,\{d_l\}\} \cup \{\{\alpha_0\},...,. \{\alpha_r\}\} \cup \{f_{\alpha_m}(n): m=0,..,r \text{ and } n \in U_{\alpha_m} \},
\]
we shall define a homomorphism $\phi_0: \text{span}(E_0) \to 2$ so that, for every $j=0,...,l$,
\[
\phi_0(\{d_j\})=F(d_j)
\]
and, for each $i=0,...,r$,
\[
\phi_0(\{\alpha_i\})=F(\alpha_i)
\]
and
\[
\phi_0(\{\alpha_i\}) = p_{\alpha_i}-\lim_{n \in U_{\alpha_i}} \phi_0(f_{\alpha_i}(n)).
\]
This can be done, since $\{\{d_0\},...,\{d_l\}\} \cup \{\{\alpha_0\},...,\{\alpha_r\}\}$ and $\{f_{\alpha_m}(n): m=0,..,r \text{ and } n \in U_{\alpha_m} \}$ are linearly independent, and thus there exists $R_0 \subset \{f_{\alpha_m}(n): m=0,..,r \text{ and } n \in U_{\alpha_m} \}$ so that $|R_0| \leq r+l+2$ and
\[
\{\{d_0\},...,\{d_l\}\} \cup \{\{\alpha_0\},...,\{\alpha_r\}\} \cup (\{f_{\alpha_m}(n): m=0,..,r \text{ and } n \in U_{\alpha_m} \} \setminus R_0)
\]
is linearly independent. 

We shall now define recursively homomorphisms $\phi_k: \text{span}(E_0 \cup \{f_{\alpha_m}(n): r<m \leq k+r \text{ and } n \in U_{\alpha_m}\} \cup \{\{\alpha_i\}: r<i\leq k+ r \}) \to 2$ satisfying that, for each $k \in \omega$,
\begin{enumerate}[1)]
\item $\phi_0$ is the homomorphism defined above;
\item $ \phi_k(\{\alpha_{k+r}\})=p_{\alpha_{k+r}}-  \lim_{n \in U_{\alpha_{k+r}}} \phi_{k}(f_{\alpha_{k+r}} (n))$;
\item $\phi_{k+1}$ extends $\phi_{k}$.
\end{enumerate}
Suppose that, for $N \in \omega$, we have defined homomorphisms $\phi_0,...,\phi_N$ satisfying 1), 2) and 3). By Lemma \ref{novo}, there is a finite subset $R_{N+1}$ of $\{f_{\alpha_{r+N+1}}(n): n \in U_{\alpha_{r+N+1}}\}$ such that
\begin{align*}
\text{span}(\{\{d_0\},...,\{d_l\}\} &\cup \{\{\alpha_0\},...,\{\alpha_{r+N}\}\} \cup  \{f_{\alpha_m}(n): 0 \leq m \leq r+N \text{ and } n \in U_{\alpha_m}\}) \bigcap \\
&\text{ span}(\{f_{\alpha_{r+N+1}}(n): n \in U_{\alpha_{r+N+1}} \} \setminus R_{N+1} ) = \{\emptyset\}.
\end{align*}
Therefore, we may define the homomorphism $\phi_{N+1}$ to be equal to $\phi_N$ in $\text{span}(\{\{d_0\},...,\{d_l\}\} \cup \{\{\alpha_0\},...,\{\alpha_{r+N}\}\} \cup \{f_{\alpha_m}(n): 0 \leq m \leq r+N \text{ and } n \in U_{\alpha_m}\})$ and so that
\[
\phi_{N+1}(\{\alpha_{r+N+1}\})=  p_{\alpha_{r+N+1}}-\lim_{n \in U_{\alpha_{r+N+1}}} \phi_{N+1}(f_{\alpha_{r+N+1}}(n)).
\]

Thus, we have proved that there exists homomorphisms $\phi_k$ satisfying 1), 2) and 3) for every $k \in \omega$. If $\phi$ is any homomorphism defined in $[D]^{<\omega}$ extending $ \bigcup_{k \in \omega} \phi_k$, then
\begin{enumerate}[i)]
    \item $\forall \alpha \in D \cap I$, $\phi(\{\alpha\})=p_{\alpha}-\lim_{n \in U_{\alpha}} \phi(f_{\alpha}(n))$;
    \item $\forall d \in D_0$, $\phi(\{d \}) = F(d)$,
\end{enumerate}
as we want.
\end{proof}

\begin{Rem}\label{remm} Note that the homomorphism $\phi$ given by the previous lemma may be defined satisfying additional properties, since we have freedom in an infinite subset of $\{f_{\alpha_k}(n): n \in U_{\alpha_k}\}$, for each $k \in \omega$. For instance, given $\alpha \in D \cap I$, we can choose $\phi$ satisfying also that
\[
\forall i \in 2, |\{n \in \omega: \phi(f_{\alpha}(n)) = i\}| = \omega.
\]
In fact, homomorphisms satisfying this property were constructed in \cite{Michael}.
\end{Rem}

\section{A countably compact group without non-trivial convergent sequences of size $2^{\mathfrak{c}}$}

\begin{Teo}
There is a Hausdorff countably compact topological Boolean group of size $2^{\mathfrak{c}}$ without non-trivial convergent sequences.
\end{Teo}
\begin{proof}
We shall construct a topology on $[2^\mathfrak{c}]^{<\omega}$ as follows. 

Following \cite{Michael}, we fix an indexed family $\{f_{\alpha}: \alpha \in [ \omega, 2^\mathfrak{c})\} \subset \left([2^{\mathfrak{c}}]^{<\omega} \right)^{\omega}$ of one-to-one sequences such that
\begin{enumerate}[1)]\label{deffun}
    \item for every infinite $X \subset [2^{\mathfrak{c}}]^{< \omega}$, there is an $\alpha \in [\omega, 2^{\mathfrak{c}})$ with $\rng(f_{\alpha}) \subset X$;
    \item each $f_{\alpha}$ is a sequence of linearly independent elements;
    \item $\rng(f_{\alpha}) \subset [\alpha]^{< \omega}$ for every $\alpha \in [\omega, 2^{\mathfrak{c}})$.
    \end{enumerate}

Let $\{p_{\alpha}: \alpha < 2^{\mathfrak{c}}\} \subset \omega^*$ be the family of free ultrafilters given by Proposition \ref{suitable}. Define, for each $\phi \in \homm$, its
extension  $\bar{\phi} \in \Hom$ recursively, by putting, for every $\alpha \in [\omega, 2^{\mathfrak{c}})$,
\begin{equation} \label{defi}
\bar{\phi}(\{\alpha\}) = p_{\alpha } -\lim \bar{\phi}(f_{\alpha}(n)).
\end{equation}
Let $\tau$ be the weakest (group) topology on $[2^{\mathfrak{c}}]^{< \omega}$ making all $\bar{\phi}$ continuous ($\phi \in \homm$). For every $\alpha \in [\omega, 2^{\mathfrak{c}})$, it follows that
\[
\{\alpha\} = p_{\alpha}- \lim_{n \in \omega} f_{\alpha}(n),
\]
since the topology is generated by finite intersections of inverse images of $\bar{\phi}$ functions, which satisfy (\ref{defi}). Therefore, as the family $\{f_{\alpha}: \alpha \in [ \omega, 2^{\mathfrak{c}})\}$ satisfies 1), the topological space $([2^{\mathfrak{c}}]^{<\omega},\tau)$ is countably compact.

Next we introduce the notion of \textit{suitably closed} set relative to this construction\footnotemark.
\begin{Def}[\cite{Michael}]
A set $D \in [2^{\mathfrak{c}}]^\omega$ is called \textit{suitably closed} if $\omega \subset D$ and $\bigcup_{n \in \omega} f_{\alpha}(n) \subset D$, for every $\alpha \in D \setminus \omega$.
\end{Def}
\footnotetext{The idea of suitably closed sets already appeared in \cite{koszmider&tomita&watson}, without using a name. Many subsequent works that used Martin's Axiom for countable posets and selective ultrafilters also used this idea.}
The topology also makes the topological group Hausdorff. Indeed, given $x \in [2^{\mathfrak{c}}]^{< \omega} \setminus \{\emptyset\}$, let $D \in [2^\mathfrak{c}]^{\omega}$ be a suitably closed set so that $x \subset D$. We may use Lemma \ref{Homs} to construct a homomorphism $\psi:[D]^{<\omega} \to 2$ satisfying (\ref{defi}) for each $\alpha \in D \setminus \omega$ and such that $\psi(x) = 1$. Hence, if $\Phi \doteq \restr{\psi}{[\omega]^{<\omega}}$, the homomorphism $\overline{\Phi} \in \Hom$ is such that $\overline{\Phi}(x) =1$, since $\restr{\overline{\Phi}}{[D]^{<\omega}}= \psi$\footnotemark.
\footnotetext{Here we justify the Hausdorff property of the group directly, unlike \cite{Michael} does.}

To finish, we enunciate the following lemma, which is used to show that the topological space $([2^{\mathfrak{c}}]^{<\omega},\tau)$ does not contain non-trivial convergent sequences. Versions of this result were already used in previous articles, but we enunciate here the version that appears in \cite{Michael}. 

\begin{Lema}\label{Lema}
If, for every $D \in [2^{\mathfrak{c}}]^{\omega}$ suitably closed and $\alpha \in D \setminus \omega$, there is $\psi \in \text{Hom}([D]^{< \omega},2)$ such that
\begin{enumerate}[(1)]
    \item $\forall \beta \in D \setminus \omega$, $\psi(\{\beta\}) = p_{\beta}-\lim_{n \in \omega} \psi(f_{\beta}(n))$;
    \item $\forall i \in 2$, $|\{n \in \omega: \psi(f_{\alpha}(n))=i\}|= \omega$,
\end{enumerate}
then the topology defined above on $[2^{\mathfrak{c}}]^{< \omega}$ does not contain non-trivial convergent sequences.
\end{Lema}
\noindent The hypothesis of the lemma above are satisfied, due to Lemma \ref{Homs} and Remark \ref{remm}, following it. 
\end{proof}
\section{A selectively pseudocompact group which is not countably pracompact}

\begin{Teo}\label{pracompacto}
There is a Hausdorff selectively pseudocompact group which is not countably pracompact.
\end{Teo}
\begin{proof}

Let $\{p_{\xi}: \xi< \mathfrak{c}\}$ be the family of free ultrafilters given by Proposition \ref{primeiro}. We fix at the beginning a function $F : \mathfrak{c} \times \mathfrak{c} \to 2$ so that:
\begin{itemize}
    \item For every $A \in [\mathfrak{c}]^{\omega}$, there exists $\beta \in \mathfrak{c}$ such that $F(A \times \{\beta\}) = 1$.
    \item For every $B \in [\mathfrak{c}]^{\omega}$, there exists $\alpha_0 \in \mathfrak{c}$ and $\alpha_1 \in \mathfrak{c}$ such that $F(\{\alpha_0\} \times B) = 0$ and $F(\{\alpha_1\} \times B) = 1$.
\end{itemize}

Let $(J_{\beta})_{\beta < \mathfrak{c}}$ be a partition of $\mathfrak{c}$ such that, for each $\beta< \mathfrak{c}$, $|J_{\beta}|= \mathfrak{c}$. Consider also, for each $\beta< \mathfrak{c},$ a partition $\{J_{\beta}^1, J_{\beta}^2 \}$ of $J_{\beta}$ satisfying that $|J_{\beta}^1|=|J_{\beta}^2|=\mathfrak{c}$. We suppose that the initial $\omega$ elements of $J_{\beta}$ are in $J_{\beta}^1$, for every $\beta < \mathfrak{c}$. Given $\beta< \mathfrak{c}$, let $\{f_{\xi}^{\beta}: \xi \in J^2_{\beta}\}$ be a family of functions $f^{\beta}_{\xi}: \omega \to [J_{\beta}]^{< \omega}$ such that 
\begin{enumerate}[1)]
    \item each $f_{\xi}^{\beta}$ is an one-to-one enumeration of linearly independent elements of $[J_{\beta}]^{< \omega}$;
    \item for every infinite $X \subset [J_{\beta}]^{<\omega}$, there exists $\xi \in J_{\beta}^{2}$ such that $\text{rng}(f_{\xi}^{\beta}) \subset X$;
    \item for every $\xi \in J_{\beta}^2$, $\text{rng}(f_{\xi}^{\beta}) \subset [\xi]^{<\omega}$.
    \end{enumerate}
From now on, we will omit the superscript of $f_{\xi}^{\beta}$, since for each $\xi \in \bigcup_{\beta < \mathfrak{c}} J_{\beta}^2$, there is a unique $\beta< \mathfrak{c}$ such that $\xi \in J_{\beta}$. We also define the sets $J_1 \doteq \bigcup_{\beta < \mathfrak{c}} J_{\beta}^1$ and $J_2 \doteq \bigcup_{\beta < \mathfrak{c}} J_{\beta}^2$. Lastly, we fix another partition $(I_{\alpha})_{\alpha < \mathfrak{c}}$ of $\mathfrak{c}$ satisfying that $|I_{\alpha}| = \mathfrak{c}$, for every $\alpha< \mathfrak{c}$. 

We shall now define which are the suitably closed sets of this construction.

\begin{Def}
A set $A \in [\mathfrak{c}]^{\omega}$ is \textit{suitably closed} if, for every $\xi \in A \cap J_2$, we have $\bigcup_{n \in \omega}f_{\xi}(n) \subset A$.
\end{Def}

Let $\mathcal{A}$ be the set of all homomorphisms $\sigma: [A]^{< \omega} \to 2$, with $A \in [\mathfrak{c}]^{\omega}$ suitably closed, satisfying that
\[
\sigma(\{\xi\}) = p_{\xi}-\lim_{n \in \omega} \sigma(f_{\xi}(n)),
\]
for every $\xi \in A \cap J_2$.
Enumerate $\mathcal{A}$ by $\{\sigma_{\mu}: \omega \leq \mu < \mathfrak{c}\}$ and, without loss of generality, we may assume that  $\bigcup \text{dom}(\sigma_{\mu}) \subset \mu$\footnotemark, for each $\omega \leq \mu < \mathfrak{c}$. We may also suppose that for each $\sigma \in \mathcal{A}$ and $\alpha< \mathfrak{c}$, there exists $\mu \in I_{\alpha}$ so that $\sigma_{\mu}= \sigma$. Now, for each $\mu \in [\omega, \mathfrak{c})$, we shall construct a suitable homomorphism $\overline{\sigma_{\mu}}:[\mathfrak{c}]^{< \omega} \to 2$, as done below. 
\footnotetext{In case of  a homomorphism $\sigma:[C]^{<\omega}\to 2$, with $C \subset \mathfrak{c}$, note that $\bigcup \text{dom}(\sigma)= C$.}

Firstly, for each $\beta< \mathfrak{c}$, we enumerate all functions $g: S \to 2$ with $S \in [\mathfrak{c}]^{<\omega}$ by $\{g_{\xi}^{\beta}: \xi \in J_{\beta}^1 \}$ and, without loss of generality, we may assume that $\text{dom}(g^{\beta}_{\xi}) \subset \xi$, for every $\xi \in J_{\beta}^1$, and that for each $g :S \to 2$ as above, $|\{\xi \in J_{\beta}^1:g_{\xi}^{\beta} = g\}| = \mathfrak{c}$. As done before, from now on we omit the superscript of $g_{\xi}^{\beta}$, since for each $\xi \in J_1$, there is a unique $\beta< \mathfrak{c}$ such that $\xi \in J_{\beta}^1$. 

Given $\omega \leq \mu < \mathfrak{c}
$, we start defining an auxiliary homomorphism $\psi_{\mu}:[\mathfrak{c}]^{<\omega} \to 2$, extending $\sigma_{\mu}$. First, if $\xi < \mathfrak{c}$ is such that $\{\xi\} \in \text{dom}(\sigma_\mu)$, we put $\psi_{\mu}(\{\xi\})= \sigma_{\mu}(\{\xi\})$. Otherwise, we have a few cases to consider: firstly, for every $\xi \in J_1$, we put $\psi_{\mu}(\{\xi\}) = g_{\xi}(\mu)$ if $\mu \in \text{dom}(g_{\xi})$ and $\psi_{\mu}(\{\xi\})=0$ if $\mu \notin \text{dom}(g_{\xi})$; for elements $\xi \in J_2$, we define $\psi_{\mu}$ recursively, by putting
\[
\psi_{\mu}(\{\xi\})= p-\lim_{n \in \omega}\psi_{\mu}(f_{\xi}(n)).
\]
Note that since $\{\{\xi\}: \xi < \mathfrak{c}\}$ is a base for $[\mathfrak{c}]^{<\omega}$, the definition above uniquely extends each $\sigma_{\mu}$ to a homomorphism $\psi_{\mu}:[\mathfrak{c}]^{<\omega} \to 2$, which satisfy the previous equation for every $\xi \in J_2$, by construction. 

Now, to obtain the homomorphisms $\overline{\sigma_{\mu}}:[\mathfrak{c}]^{<\omega} \to 2$, we shall make some modifications to the homomorphisms $\psi_{\mu}$ defined before, in the following way. For every $\omega \leq \mu < \mathfrak{c}$ and $\xi<\mathfrak{c}$, we have that $\mu \in I_{\alpha}$ and $\xi \in J_{\beta}$ for unique $\alpha, \beta < \mathfrak{c}$, thus we put
\[
\begin{cases}
\overline{\sigma_{\mu}}(\{\xi\}) = 0, & \text{if }  F(\alpha, \beta)=0\\
        \overline{\sigma_{\mu}}(\{\xi\}) = \psi_{\mu}(\{\xi\}),    & \text{if } F(\alpha, \beta)=1.
\end{cases}\]
Again, in this way we define uniquely a homomorphism $\overline{\sigma_{\mu}}:[\mathfrak{c}]^{<\omega} \to 2$, for each $\mu < \mathfrak{c}$. Note that, for every $\xi \in J_2$ and $\omega \leq \mu <\mathfrak{c}$, 
\[
\overline{\sigma_{\mu}}(\{\xi\}) = p_{\xi}-\lim_{n \in \omega} \overline{\sigma_{\mu}}(f_{\xi}(n)).
\]
In fact, if $\mu \in I_{\alpha}$ and $\xi \in J_{\beta}^2$ are such that $F(\alpha, \beta)=0$, then $\overline{\sigma_{\mu}}(\{\xi\})=\overline{\sigma_{\mu}}(f_{\xi}(n))=0$ for every $n \in \omega$. Otherwise, $\overline{\sigma_{\mu}}(\{\xi\})= \psi_{\mu}(\{\xi\})$ and $\overline{\sigma_{\mu}}(f_{\xi}(n)) = \psi_{\mu}(f_{\xi}(n))$, for every $n \in \omega$. Furthermore, note that $\overline{\sigma_{\mu}}$ is non-trivial for every $\omega \leq \mu < \mathfrak{c}$. Indeed, given $\mu \in I_{\alpha}$, take $\beta< \mathfrak{c}$ such that $F(\alpha,\beta)=1$ and let $\xi \in J_{\beta}^1 \setminus \bigcup \text{dom}(\sigma_{\mu})$ be such that $g_{\xi}$ is the function which has $\text{dom}(g_{\xi})=\{\mu\}$ and $g_{\xi}(\mu)=1$. Hence, we have that $\overline{\sigma_{\mu}}(\{\xi\})= g_{\xi}(\mu)=1$.

Let now $\overline{\mathcal{A}} \doteq \{\overline{\sigma_{\mu}}: \omega \leq \mu < \mathfrak{c}\}$ and $\tau$ be the weakest (group) topology on $[\mathfrak{c}]^{<\omega}$ making every homomorphism in $\overline{\mathcal{A}}$ continuous. First, notice that $([\mathfrak{c}]^{<\omega},\tau)$ is Hausdorff. In fact, given $x \in [\mathfrak{c}]^{<\omega}\setminus \{\emptyset\}$, let $D \in [\mathfrak{c}]^{\omega}$ be a suitably closed set so that $x \subset D$, and let $\alpha < \mathfrak{c}$ be such that $F(\{\alpha\} \times D)$=1. According to Lemma \ref{Homs}, there exists $\sigma:[D]^{<\omega}\to 2$, $\sigma \in \mathcal{A}$, such that $\sigma(x)=1$ and, by construction, there exists $\mu_0 \in I_{\alpha}$ such that $\sigma_{\mu_0}=\sigma$. Hence, we have that $\overline{\sigma_{\mu_0}}(x)=1$.

\begin{Cla}
$([\mathfrak{c}]^{<\omega},\tau)$ is a selectively pseudocompact group.
\end{Cla}
\begin{proof}[Proof of the claim.]\renewcommand{\qedsymbol}{$\blacksquare$}
Let $\{U_n: n \in \omega\}$ be a sequence of nonempty open sets in the group. For each $U_n$, we fix a function $g_n:S_n \to 2$, with $S_n \in [\mathfrak{c}]^{<\omega} \setminus \{\emptyset\}$, so that
\[
U_n \supset \bigcap_{\mu \in S_n} \overline{\sigma_{\mu}}^{^{-1}}(g_n(\mu)).
\]

Thus, for each $n \in \omega$, we may define the set $C_n \doteq \{\alpha < \mathfrak{c}: \mu \in I_{\alpha} \text{ for some } \mu \in S_n\}$, and also $C \doteq \bigcup_{n \in \omega} C_n \in [\mathfrak{c}]^{\omega}$. By the property of $F$ function, there exists $\beta \in \mathfrak{c}$ such that $F(C \times \{\beta\})=1$. For each $n \in \omega$, let $\xi_n \in J_{\beta}^1 \setminus \bigcup_{\mu \in S_n} \big( \cup \text{dom}(\sigma_{\mu}) \big)$ be such that $g_{\xi_n}=g_n$. We may choose such elements $\xi_n$ pairwise distinct. By the way we have defined the homomorphisms which generate the topology, for every $\mu \in S_n$, $\overline{\sigma_{\mu}}(\{\xi_n\})=\psi_{\mu}(\{\xi_n\})=g_n(\mu)$, and therefore $\{\xi_n\} \in U_n$, for every $n \in \omega$. 

Now, let $\xi \in J_{\beta}^2$ be such that $\text{rng}(f_{\xi}) \subset \{\{\xi_n\}: n \in \omega\}$. Since
\begin{equation}\label{claim1}
\overline{\sigma_{\mu}}(\{\xi\}) = p_{\xi}-\lim_{n \in \omega} \overline{\sigma_{\mu}}(f_{\xi}(n)),
\end{equation}
for every $\omega \leq \mu < \mathfrak{c}$, we have that $\{\xi\}$ is an accumulation point of  $\{\{\xi_n\}: n \in \omega\}$, ending the proof. 
\end{proof}

\begin{Cla}
$([\mathfrak{c}]^{<\omega}, \tau)$ is not a countably pracompact group.
\end{Cla}
\begin{proof}[Proof of the claim]\renewcommand{\qedsymbol}{$\blacksquare$}
Suppose that $Z$ is a subset of $G$ that is dense. We shall construct a sequence $\{t_{n}: n \in \omega\} \subset Z$ that does not have an accumulation point in $([\mathfrak{c}]^{<\omega},\tau)$. Such sequence shall satisfy that
\begin{equation}\label{isnot}
\text{SUPP}(t_{n})  \setminus \left( \bigcup_{m< n} \text{SUPP}(t_{m})\right) \neq \emptyset, \end{equation} for every $n \in \omega$, where, for each $D \in [\mathfrak{c}]^{<\omega}$, we define
\[
\text{SUPP}(D) \doteq \{\beta < \mathfrak{c}: J_{\beta} \cap D \neq \emptyset\}.
\]

First, fix $t_0 \in Z$ arbitrarily, and suppose that, for $k>0$, we have defined $\{t_{n}: n < k \} \subset Z$ satisfying equation (\ref{isnot}) for every $n<k$. We claim that $B \doteq \bigcup_{z \in Z} \text{SUPP}(z)$ cannot be a countable set. Indeed, if $B$ is countable, by construction of $F$ function, there exists an $\alpha \in \mathfrak{c}$ so that $F(\{\alpha\}\times B)=0$. Therefore, given $\mu \in I_{\alpha}$, $\overline{\sigma_{\mu}}(z)=0$ for every $z \in Z$, and thus $Z$ cannot be dense in $G$. Hence, there exists $t_{k} \in Z$ such that 
\[
\text{SUPP}(t_{k})  \setminus \left( \bigcup_{m<k} \text{SUPP}(t_{m})\right)  \neq \emptyset,
\]
which proves the existence of such sequence $\{t_{n}: n \in \omega\} \subset Z$.

Now we shall show that, for each $x \in [\mathfrak{c}]^{< \omega}$, $x$ is not an accumulation point of $\{t_{n}: n \in \omega\}$. First, note that there exists $k_0 \in \omega$ such that, for every $n \geq k_0$,
\[
\text{SUPP}(t_n) \setminus \left( \bigcup_{m<n} \text{SUPP}(t_m) \cup \text{SUPP}(x)   \right) \neq \emptyset.
\]
In fact, since $\text{SUPP}(x)$ is finite and (\ref{isnot}) holds, there cannot be infinite elements $t_n$ such that $\text{SUPP}(t_n) \subset \bigcup_{m<n}\text{SUPP}(t_m) \cup \text{SUPP}(x)$. 

Let
\[
F_0 \doteq \bigcup_{m<k_0}\text{SUPP}(t_m) \cup \text{SUPP}(x)
\]
and, for $i>0$,
\[
F_i \doteq \text{SUPP}(t_{k_0+i-1}) \setminus \left(\bigcup_{m< k_0+i-1} \text{SUPP}(t_m) \cup \text{SUPP}(x)  \right).
\]
Define also, for each $i \in \omega$,
\[
D_i \doteq \{\xi \in \bigcup_{n \in \omega} t_n \cup x: \text{SUPP}(\{\xi\}) \in F_i\},
\]
and let $A_i$ be a suitably closed set containing $D_i$ such that $\text{SUPP}(A_i) = \text{SUPP}(D_i)$. Since $(F_i)_{i \in \omega}$ is a family of pairwise disjoint sets, we have that $(A_i)_{i \in \omega}$ is also a family of pairwise disjoint sets.

According to Lemma \ref{Homs}, we may define a homomorphism $\theta_0 : [A_0]^{<\omega} \to 2$ such that $\theta_0 \in \mathcal{A}$ and $\theta_0(x)=0$. For $k>0$, suppose that we have constructed a set of homomorphisms $\{\theta_i: i <k\} \subset \mathcal{A}$ such that
\begin{enumerate}[i)]
   \item $\theta_0(x)=0$.
    \item $\theta_i$ is a homomorphism defined in $\big[ \bigcup_{j \leq i} A_j \big]^{<\omega}$ taking values in $2$, for each $i<k$.
    \item $\theta_i$ extends $\theta_{i-1}$ for each $0<i<k$.
    \item $\theta_{i}(t_{k_0+j})=1$ for each $0<i<k$ and $j=0,...,i-1$.
\end{enumerate}
Let $A_k$ be a suitably closed set containing $D_k$. Again by Lemma \ref{Homs}, we may define a homomorphism $\psi_k:[A_k]^{<\omega} \to 2$ so that $\psi_k \in \mathcal{A}$ and 
\[\psi_{k}\Big(t_{k_0+k-1} \setminus \bigcup_{j < k} D_j \Big) + \theta_{k-1}\Big(t_{k_0+k-1} \cap \bigcup_{j < k} D_j \Big) = 1.\]
Now, since $A_k \cap \bigcup_{i < k} A_i = \emptyset$, we may define a homomorphism $\theta_k : \big[\bigcup_{j \leq k} A_j \big]^{<\omega} \to 2$ extending both $\theta_{k-1}$ and $\psi_k$. Then, by construction, we have that $\theta_k(x)=0$ and $\theta_k(t_{k_0+j})=1$ for every $j=0,...,k-1$. Also, it follows that $\theta_k \in \mathcal{A}$, since $\psi_k \in \mathcal{A}$ and $\theta_i \in \mathcal{A}$ for every $i<k$. Therefore, there exists a family of homomorphisms $\{\theta_k: k \in \omega\} \subset \mathcal{A}$ satisfying i)-iv) for every $k \in \omega$.

Letting $A \doteq \bigcup_{n \in \omega} A_n$ and $\theta \doteq \bigcup_{n \in \omega} \theta_n$, the homomorphism $\theta:[A]^{<\omega} \to 2$ satisfy that $\theta \in \mathcal{A}$, since each $\theta_n \in \mathcal{A}$. Also, $\theta(x)=0$ and $\theta(t_{k_0+j})=1$ for every $j \in \omega$. Let $\alpha < \mathfrak{c}$ be such that $F \Big( \{\alpha\} \times \bigcup_{n \in \omega} \text{SUPP}(t_n) \cup \text{SUPP}(x) \Big)=1$ and $\mu \in I_{\alpha}$ be such that $\sigma_{\mu} = \theta$. Then, $\overline{\sigma_\mu}: [\mathfrak{c}]^{<\omega} \to 2$ satisfy that $\overline{\sigma_{\mu}}(t_n)=1$ for each $n \in \omega$, and $\overline{\sigma_{\mu}}(x)=0$. Therefore, $x \in [\mathfrak{c}]^{<\omega}$, which was chosen arbitrarily, is not an accumulation point of $\{t_n: n \in \omega\}$, hence $([\mathfrak{c}]^{<\omega},\tau)$ is not countably pracompact.
\end{proof}

\end{proof}

\section{A group which is not countably pracompact and has all powers selectively pseudocompact, from a single selective ultrafilter}

Assuming the existence of a selective ultrafilter $p$, we may use the same construction that was done in the previous section to show that there exists a selectively $p-$pseudocompact group which is not countably pracompact. Since selective $p-$pseudocompactness is productive and implies selective pseudocompactness, we will obtain a group which has all powers selectively pseudocompact and is not countably pracompact. In order to replace Lemma \ref{Homs}, we use some versions of results proved in Tomita, Garcia-Ferreira and Watson's paper \cite{selective}. 

Following the proof of Theorem \ref{pracompacto}, we consider the same function $F : \mathfrak{c} \times \mathfrak{c} \to 2$, families $(I_{\alpha})_{\alpha<\mathfrak{c}}$, $(J_{\beta})_{\beta < \mathfrak{c}}$ and partition $\{J_{\beta}^1, J_{\beta}^2 \}$ of $J_{\beta}$, for each $\beta< \mathfrak{c}$. Using a similar proof to \textit{Lemma 2.1} of \cite{selective}, one may show the following result.
\begin{Lema}[\cite{selective}, Lemma 2.1]\label{funcoes}
If $p \in \omega^*$ is a selective ultrafilter, then, for each $\beta< \mathfrak{c}$, there exists a family of one-to-one functions $\{f_{\xi}: \xi \in J_{\beta}^2\} \subset ([J_{\beta}]^{<\omega})^{\omega}$ such that
\begin{enumerate}[i)]
\item $\bigcup_{n \in \omega} f_{\xi}(n) \subset \text{max}\{\omega, \xi\}$, for every $\xi \in J_{\beta}^2$.
\item $\{[f_{\xi}]_p: \xi \in J_{\beta}^2\} \cup \{[\vec{\mu}]_{p}: \mu \in J_{\beta}\}$ is a base for $([J_{\beta}]^{<\omega})^{\omega}/p$.
\item For every one-to-one function $g \in ([J_{\beta}]^{<\omega})^{\omega}$, there exists distinct $\xi_0, \xi_1 \in J_{\beta}^2$ and two increasing sequences of positive integers $(n^0_k)_{k< \omega}$ and $(n^1_k)_{k < \omega}$ such that $f_{\xi_i}(k)=g(n^i_k)$, for every $k < \omega$ and $i \in 2$.
\end{enumerate}
\end{Lema}

In what follows, we fix a family $\{f_{\xi}: \xi \in J_2\}$ so that, for each $\beta < \mathfrak{c}$, $(f_{\xi})_{\xi \in J_{\beta}^2}$ satisfy the three properties stated in the previous result. In this case, it is not hard to show that $\{[f_{\xi}]_p: \xi \in J_2\} \cup \{[\vec{\mu}]_p: \mu < \mathfrak{c}\}$ is linearly independent in $([\mathfrak{c}]^{<\omega})^{\omega}/p$ and hence one may repeat the proof of \textit{Lemma 2.3} in \cite{selective} to show that\footnotemark:
\footnotetext{To adapt the proof, we consider $F_0 \doteq E_0$ and, for each $n \geq 0$, $F_{n+1} \doteq F_n \cup \Big[ \bigcup_{\xi \in F_n \cap J_2} \bigcup_{m \leq n}f_{\xi}(m)\Big]$. The family $\{E_i: 0<i< \omega\}$ will be a subsequence of $\{F_i: 0<i< \omega\}$. } 
\begin{Lema}[\cite{selective}, Lemma 2.3]
Let $p \in \omega^*$ be selective. For every $E_0 \in [\mathfrak{c}]^{<\omega}\setminus \{\emptyset\}$, there are $\{b_i: i < \omega\} \in p$ and $\{E_i: 0<i<\omega\} \subset [\mathfrak{c}]^{<\omega}$ such that
\begin{enumerate}[1)]
    \item $E_i \cup \big[ \bigcup_{\xi \in E_i \cap J_2}f_{\xi}(b_i)\big] \subset E_{i+1}$, for every $i < \omega$; 
    \item $\{f_{\xi}(b_i):\xi \in E_i \cap J_2\} \cup \{\{\mu\}: \mu \in E_i\}$ is linearly independent, for every $i < \omega$.
\end{enumerate}
\end{Lema}

Now, we can show the lemma that will replace Lemma \ref{Homs}. A similar result was also proved in \cite{selective} (see \textit{Example 2.4}) but we repeat the proof here, for the sake of completeness. 
\begin{Lema}\label{novohom}
Let $p \in \omega^*$ be a selective ultrafilter and $D \in [\mathfrak{c}]^{\omega}$ be such that, for every $\alpha \in D \cap J_2$, $\bigcup_{n \in \omega}f_{\alpha}(n) \subset D$. Then, for each $E_0 \in [D]^{<\omega} \setminus \{\emptyset\}$, there exists a homomorphism $\Phi: [D]^{<\omega} \to 2$ such that
\begin{enumerate}
    \item $\Phi(\{\xi\})=p-\lim_{n \in \omega} \Phi(f_{\xi}(n))$, for every $\xi \in D \cap J_2$.
    \item $\Phi(E_0)=1$.
\end{enumerate}
\end{Lema}
\begin{proof}
By applying the previous lemma to $E_0$, we obtain $\{b_i: i< \omega\} \in p$ and $\{E_i: 0<i<\omega\} \subset [\mathfrak{c}]^{<\omega}$ such that\footnotemark
\footnotetext{Note that due to the form of $\{E_i: 0<i<\omega\}$ sets (see the previous footnote), we have that  $ \bigcup_{i \in \omega}E_i \subset D$.}
\begin{enumerate}[1)]
     \item $E_i \cup \big[ \bigcup_{\xi \in E_i \cap J_2}f_{\xi}(b_i)\big] \subset E_{i+1}$, for every $i < \omega$;
    \item $\{f_{\xi}(b_i):\xi \in E_i \cap J_2\} \cup \{\{\mu\}: \mu \in E_i\}$ is linearly independent, for every $i < \omega$.
\end{enumerate}

Since $\{f_{\xi}(b_0): \xi \in E_0 \cap J_2\} \cup \{\{\mu\}: \mu \in E_0\}$ is linearly independent and $E_0 \cup \big[ \bigcup_{\xi \in E_0 \cap J_2}f_{\xi}(b_0)\big] \subset E_{1}$, we may define a homomorphism $\Phi_1:[E_1]^{<\omega} \to 2$ such that $\Phi_1(E_0)=1$ and $\Phi_1(f_{\xi}(b_0))=\Phi_1(\{\xi\})$ for every $\xi \in J_2 \cap E_0$.

Suppose that, for $0< i < \omega$, we have defined $\Phi_i:[E_i]^{<\omega}\to 2$ so that $\Phi_i(E_0)=1$ and $\Phi_i(f_{\xi}(b_{i-1}))=\Phi_i(\{\xi\})$ for every $\xi \in J_2 \cap E_{i-1}$. Since $\{f_{\xi}(b_{i}): \xi \in E_{i} \cap J_2\} \cup \{\{\mu\}: \mu \in E_{i}\}$ is linearly independent and $E_i \cup \Big[ \bigcup_{\xi \in E_i \cap J_2}f_{\xi}(b_i) \Big]\subset E_{i+1}$, we may define a homomorphism $\Phi_{i+1}: [E_{i+1}]^{<\omega} \to 2$ extending $\Phi_i$ so that $\Phi_{i+1}(f_{\xi}(b_{i}))=\Phi_{i+1}(\{\xi\})$ for every $\xi \in E_{i} \cap J_2$. Thus, such homomorphisms $\Phi_i$ exist for every $i > 0$. 

Now let $E \doteq \bigcup_{n \in \omega} E_n$ and $\psi \doteq \bigcup_{n >0} \Phi_n : [E]^{<\omega} \to 2$. We may extend $\psi$ to a homomorphism $\Phi$ defined on $[D]^{<\omega}$ by putting $\Phi(\{\xi\})=0$ if $\xi \in J_1 \cap (D \setminus E)$ and then, recursively, 
\[
\Phi(\{\xi\}) = p-\lim_{n \in \omega} \Phi(f_{\xi}(n)),
\]
for every $\xi \in J_2 \cap (D \setminus E$). The homomorphism $\Phi:[D]^{<\omega} \to 2$ satisfy the required properties. Indeed, if $\xi \in J_2 \cap (D \setminus E)$, then (1) follows by the previous equation. On the other hand, if $\xi \in J_2 \cap E$ and $j \in \omega$ is such that $\xi \in E_j$, then $\{b_i: i \geq j\} \subset \{n \in \omega: \Phi(f_{\xi}(n))=\Phi(\{\xi\})\} $ and hence 
\[
\Phi(\{\xi\})= p-\lim_{n \in \omega} \Phi(f_{\xi}(n)).
\]
\end{proof}

Next, we show the main result of this section. Since the arguments are analogous to those in the proof of Theorem \ref{pracompacto}, we omit some details. As before, a set $A \in [\mathfrak{c}]^{\omega}$ will be called \textit{suitably closed} if, for every $\xi \in J_2 \cap A$, $\bigcup_{n \in \omega} f_{\xi}(n) \subset A$.
\begin{Teo} 
If $p \in \omega^*$ is a selective ultrafilter, there exists a Hausdorff selectively $p-$pseudocompact group which is not countably pracompact.
\end{Teo}
\begin{proof}
Let $\mathcal{A}$ be the set of all homomorphisms $\sigma: [A]^{< \omega} \to 2$, where $A$ is a suitably closed set, satisfying that
\[
\sigma(\{\xi\}) = p-\lim_{n \in \omega} \sigma(f_{\xi}(n)),
\]
for every $\xi \in A \cap J_2$.
Enumerate $\mathcal{A}$ by $\{\sigma_{\mu}: \omega \leq \mu < \mathfrak{c}\}$, assuming that $\bigcup \text{dom}(\sigma_{\mu}) \subset \mu$, for each $\omega \leq \mu < \mathfrak{c}$, and also that for each $\sigma \in \mathcal{A}$ and $\alpha< \mathfrak{c}$, there exists $\mu \in I_{\alpha}$ so that $\sigma_{\mu}= \sigma$. As before, we shall construct a suitable homomorphism $\overline{\sigma_{\mu}}:[\mathfrak{c}]^{< \omega} \to 2$, for each $\omega \leq \mu < \mathfrak{c}$.

We consider the same enumeration $\{g_{\xi}: \xi \in J_1\}$ of all functions $g: S \to 2$ with $S \in [\mathfrak{c}]^{<\omega}$ fixed in Theorem \ref{pracompacto}. For each $\omega \leq \mu < \mathfrak{c}
$, we define the auxiliary homomorphism  $\psi_{\mu}:[\mathfrak{c}]^{<\omega} \to 2$, extending $\sigma_{\mu}$, in the following way. If $\xi < \mathfrak{c}$ is such that $\{\xi\} \in \text{dom}(\sigma_\mu)$, we put $\psi_{\mu}(\{\xi\})= \sigma_{\mu}(\{\xi\})$. Otherwise, we have a few cases to consider: firstly, for every $\xi \in J_1$, we put $\psi_{\mu}(\{\xi\}) = g_{\xi}(\mu)$ if $\mu \in \text{dom}(g_{\xi})$ and $\psi_{\mu}(\{\xi\})=0$ if $\mu \notin \text{dom}(g_{\xi})$; for elements $\xi \in J_2$, we define $\psi_{\mu}$ recursively, by putting
\[
\psi_{\mu}(\{\xi\})= p-\lim_{n \in \omega}\psi_{\mu}(f_{\xi}(n)).
\]
It is not hard to see that the homomorphism $\psi_{\mu}:[\mathfrak{c}]^{<\omega} \to 2$ satisfy the previous equation for every $\xi \in J_2$.

Now, for every $\omega \leq \mu < \mathfrak{c}$ and $\xi<\mathfrak{c}$, we have that $\mu \in I_{\alpha}$ and $\xi \in J_{\beta}$ for unique $\alpha, \beta < \mathfrak{c}$, hence we may put
\[
\begin{cases}
\overline{\sigma_{\mu}}(\{\xi\}) = 0, & \text{if }  F(\alpha, \beta)=0\\
        \overline{\sigma_{\mu}}(\{\xi\}) = \psi_{\mu}(\{\xi\}),    & \text{if } F(\alpha, \beta)=1.
\end{cases}\]
As before, in this way we define uniquely a non-trivial homomorphism $\overline{\sigma_{\mu}}:[\mathfrak{c}]^{<\omega} \to 2$ so that, for every $\xi \in J_2$ and $\omega \leq \mu <\mathfrak{c}$, 
\[
\overline{\sigma_{\mu}}(\{\xi\}) = p-\lim_{n \in \omega} \overline{\sigma_{\mu}}(f_{\xi}(n)).
\]

Let now $\overline{\mathcal{A}} \doteq \{\overline{\sigma_{\mu}}: \omega \leq \mu < \mathfrak{c}\}$ and $\tau$ be the weakest (group) topology on $[\mathfrak{c}]^{<\omega}$ making every homomorphism in $\overline{\mathcal{A}}$ continuous. The topological group $([\mathfrak{c}]^{<\omega},\tau)$ is Hausdorff. Indeed, given $x \in [\mathfrak{c}]^{<\omega}\setminus \{\emptyset\}$, let $D \in [\mathfrak{c}]^{\omega}$ be a suitably closed set so that $x \subset D$, and let $\alpha < \mathfrak{c}$ be such that $F(\{\alpha\} \times D)$=1. According to Lemma \ref{novohom}, there exists $\sigma:[D]^{<\omega}\to 2$, $\sigma \in \mathcal{A}$, such that $\sigma(x)=1$ and, by construction, there exists $\mu_0 \in I_{\alpha}$ such that $\sigma_{\mu_0}=\sigma$. Hence, $\overline{\sigma_{\mu_0}}(x)=1$.

\begin{Cla}
$([\mathfrak{c}]^{<\omega},\tau)$ is a selectively $p-$pseudocompact group.
\end{Cla}
\begin{proof}[Proof of the claim.]\renewcommand{\qedsymbol}{$\blacksquare$}
Let $\{U_n: n \in \omega\}$ be a sequence of nonempty open sets in the group. We proceed the same way as in \textbf{Claim 1} of Theorem \ref{pracompacto} to construct a sequence of pairwise distinct elements $\{\{\xi_n\}: n \in \omega\}$ such that $\{\xi_n\} \in U_n$ for each $n \in \omega$ and such that, for some fixed $\beta<\mathfrak{c}$, $\xi_n \in J_{\beta}$ for every $n \in \omega$. 

Let $g \in ([J_{\beta}]^{<\omega})^{\omega}$ be such that $g(n)=\{\xi_n\}$, for every $n \in \omega$. Since $\{[f_{\xi}]_p: \xi \in J_{\beta}^2\} \cup \{[\vec{\mu}]_{p}: \mu \in J_{\beta}\}$ is a base for $([J_{\beta}]^{<\omega})^{\omega}/p$, there exists $\eta_0,..., \eta_k \in J_{\beta}^2$ and $E \in [J_{\beta}]^{<\omega}$ such that 
\[
[g]_p = (\triangle_{i \leq k} [f_{\eta_i}]_p) \triangle (\triangle_{\mu \in E}[\vec{\mu}]_p).
\]
Hence, there exists $B \in p$ such that, for every $n \in B$,
\[
\{\xi_n\} =  (\triangle_{i \leq k} f_{\eta_i}(n)) \triangle (\triangle_{\mu \in E} \{\mu\}) =  (\triangle_{i \leq k} f_{\eta_i}(n)) \triangle E.
\]
Therefore, since for each $i=0,...,k$ we have
\[
\{\eta_i\} = p-\lim_{n \in \omega} f_{\eta_i}(n)
\]
by construction, it follows that
\[
(\triangle_{i \leq k} \{\eta_i\}) \triangle E = p-\lim_{n \in \omega} \{\xi_n\},
\]
and thus $\{\{\xi_n\}: n \in \omega\}$ has a $p-$limit.
\end{proof}

\begin{Cla}
$([\mathfrak{c}]^{<\omega}, \tau)$ is not a countably pracompact group.
\end{Cla}
\begin{proof}[Proof of the claim]\renewcommand{\qedsymbol}{$\blacksquare$}
It is the same proof done in \textbf{Claim 2} of Theorem \ref{pracompacto}, just changing the use of Lemma \ref{Homs} for Lemma \ref{novohom}.
\end{proof}

\end{proof}

\section{Some remarks and questions}

There are still many open questions regarding the pseudocompact-like properties in topological groups. For instance, \cite{tomita1} asks:
\begin{Ques}
Is there a pseudocompact, non-selectively pseudocompact group which is connected?
\end{Ques}
\begin{Ques}
If an Abelian group admits a pseudocompact group topology, does it admit a selectively pseudocompact group topology?
\end{Ques}
\begin{Ques}
Does every compact group admit a proper dense selectively pseudocompact subgroup?
\end{Ques}

In \textbf{Claim 2} of Theorem \ref{pracompacto}, we proved that if $Z \subset ([\mathfrak{c}]^{<\omega},\tau)$ is a dense subset, then $B \doteq \bigcup_{z \in Z} \text{SUPP}(z)$ is not countable. In particular, this shows that the group we have constructed is not separable. Thus, another interesting question is:
\begin{Ques}
Is there a separable selectively pseudocompact group which is not countably pracompact?
\end{Ques}

Productivity of pseudocompact-like properties in topological groups has been widely studied in the last years. As mentioned previously, Comfort and Ross proved that the product of any family of pseudocompact groups is pseudocompact \cite{comross}, and Hrušák, van Mill, Ramos-García, and Shelah proved that there exists two countably compact topological groups whose product is not countably compact \cite{Michael}. Nevertheless, the following Comfort-like questions remain unsolved in ZFC.
\begin{Ques}
Is it true in ZFC that selective pseudocompactness is non-productive in the class of topological groups?
\end{Ques}
\begin{Ques}
Is it true in ZFC that countable pracompactness is non-productive in the class of topological groups?
\end{Ques}
\noindent Using tools outside ZFC, there are already solutions to the questions above. For instance, Garcia-Ferreira and Tomita proved that if $p$ and $q$ are non-equivalent (according to the Rudin-Keisler ordering in $\omega^*$) selective ultrafilters on $\omega$, then there are a $p$-compact group and a $q$-compact group whose product is not selectively pseudocompact \cite{tomita2}. Also, Bardyla, Ravsky, and Zdomskyy constructed, under Martin's Axiom, a Boolean countably compact topological
group whose square is not countably pracompact \cite{Bardyla}.

Convergent sequences in topological groups are related to several important concepts concerning these spaces, and hence have also been studied for a long time. We point out that it is possible to find a counterexample to the famous problem posed by Wallace \cite {Wallace} (written in the next question) inside of any non-torsion countably compact topological group without non-trivial convergent sequences (according to \cite{walmartin} and \cite{waltomita}).
\begin{Ques}[\cite{Wallace}]
Is every countably compact topological semigroup with two-sided cancellation a topological group?
\end{Ques}
 A counterexample to Wallace’s question has been called a \textit{Wallace semigroup}. Hence, a positive answer to the following question would prove the existence of a Wallace semigroup in ZFC.
\begin{Ques}
Is there in ZFC a non-torsion countably compact topological group without non-trivial convergent sequences?
\end{Ques}
\noindent The known examples of Wallace semigroups are under CH \cite{walmartin}, Martin’s Axiom for countable posets \cite{waltomita}, $\mathfrak{c}$ incomparable selective ultrafilters (according to the Rudin-Keisler ordering in $\omega^*$) \cite{madariaga-garcia&tomita} and a single selective ultrafilter \cite{boero&castro&tomita2019}. With the exception of \cite{waltomita}, the articles mentioned obtained the examples as a semigroup of a countably compact free Abelian group without non-trivial convergent sequences. We point out that Fuchs showed that a non-trivial free Abelian group does not admit a compact Hausdorff group topology, however, it is not known whether the same is true for countably compact topologies.

\bibliographystyle{plain}
\bibliography{main}
  \Addresses

\end{document}